\documentclass[leqno,12pt]{article} %leqno is the option to put formula numbers on the left side
\usepackage{amsmath, amssymb}
\usepackage{amsthm} %theorem environment option
\usepackage[all]{xy}
\usepackage{pstricks, pst-tree}
\theoremstyle{plain} %text of this environment is typesetted in italics
\newtheorem{theorem}{\indent\sc Theorem}[section]
\newtheorem{lemma}[theorem]{\indent\sc Lemma}
\newtheorem{corollary}[theorem]{\indent\sc Corollary}
\newtheorem{proposition}[theorem]{\indent\sc Proposition}

\theoremstyle{definition} %text of this environment is typesetted in roman letters
\newtheorem{definition}[theorem]{\indent\sc Definition}
\newtheorem{remark}[theorem]{\indent\sc Remark}
\newtheorem{example}[theorem]{\indent\sc Example}

\newcommand{\N}{\mathbb{N}}
\newcommand{\Z}{\mathbb{Z}}
\newcommand{\R}{\mathbb{R}}
\newcommand{\C}{\mathbb{C}}

\makeatletter
\def\address#1#2{\begingroup
\noindent\parbox[t]{7.8cm}{%
\small{\scshape\ignorespaces#1}\par\vskip1ex
\noindent\small{\itshape E-mail address}%
\/: #2\par\vskip4ex}\hfill%
\endgroup}%
\makeatother
\title{\uppercase{Coarse Dynamics and fixed point property}}
\author{
\bigskip \\
\textsc{Tomohiro Fukaya$^{*}$} %names of authors
}
\date{} %leave empty
\begin{document}

\maketitle

%%%%%%%%%%%%%%% footnote %%%%%%%%%%%%%%%%
\footnote{ %2000 MSC numbers
}
\footnote{ %key words and phrases
}
\footnote{ %acknowledgment of support etc. if any
$^{*}$The author is supported by Grant-in-Aid for JSPS Fellows 
(19$\cdot$3177) from Japan Society for the Promotion
of Science.
}
%%%%%%%%%%%%%%%%%%%%%%%%%%%%%%%%%%%%%%%%%

\begin{abstract}
We investigate the fixed point property of the group actions on a coarse
 space and its Higson corona. We deduce the coarse version of Brouwer's fixed
 point theorem.
\end{abstract}

\section{Introduction} %delete * to number this section
Let $X$ be a proper coarse space and $G$ be a 
finitely generated semi-group acting on $X$. 
Here a coarse
structure of $G$ is generated by the `diagonals'
\[
 \Delta_g = \{(h,hg)| h \in G\}
\]
as $g$ runs over a generating set for $G$. 
We will define the
`coarse action' of $G$ on $X$. In this paper, we always assume that the
coarse space is coarsely connected. The notation on the
coarse geometry are based on Chapter 2 of 
the textbook written by John Roe
\cite{MR2007488}. 
%%%% Definition of a coarse action.
\begin{definition}
An action of $G$ on $X$ is the {\itshape coarse action} if
 for each element $g$ of $G$, the map $\Psi_g \colon X \rightarrow X$
 defined by $x \mapsto g\cdot x$ is a coarse map. 
\end{definition}
%%%% Definition of orbits.
\begin{definition}
 For a point $x_0$ of $X$, we define the orbit map $\Phi_{x_0}\colon G
 \rightarrow X$ by $g \mapsto g\cdot x_0$ and also we define:\\
\indent
The orbit of $x_0$ is {\itshape proper} 
if $\Phi_{x_0}$ is a proper map.\\
\indent
The orbit of $x_0$ is {\itshape bornologous}
 if $\Phi_{x_0}$ is a bornologous map.\\
\indent
The orbit of $x_0$ is {\itshape coarse} 
if $\Phi_{x_0}$ is a coarse map.
\end{definition}
%%%%
The most typical example of the coarse action is mentioned in Lemma
\ref{lem:coarse_action}.
%%%%%%%%%%%%%%%%%%%%%%
%%%% Main Theorem %%%%
%%%%%%%%%%%%%%%%%%%%%%
\begin{theorem}%[Main Theorem]
\label{th:main_theorem}
 Assume that $G=\N^k \text{ or } \Z^k$ and $G$ acts on $X$ as a coarse action.
 Suppose that there exists a point $x_0$ of $X$ whose orbit coarse.
 Then the induced action of $G$ on the Higson corona $\nu X$ has
 a fixed point. Namely, there
 exists a point $x$ of $\nu X$ such that $g \cdot x = x$
 for any element $g \in G$.

 Moreover, Let $\overline{X}$ be a coarse compactification of $X$.
 We suppose that the action of $G$ on $X$ is continuous and extends to
 a continuous action on $\overline{X}$. Then the action on the boundary
 $\partial X = \overline{X}\setminus X$ has a fixed point.
 Namely, there exists a point $z$ of $\partial X$ such that 
 $g \cdot z = z$ for any element $g \in G$.
\end{theorem}
%\begin{remark}
% The Higson compactification $hX$ of $X$ is a coarse compactification and
% a coarse action on $X$ induces the continuous action on the Higson
%corona  $\nu X$. 
%\end{remark}
%%
\begin{example}
 Let $G$ be a finitely generated group with an element $\gamma$ of 
 infinite order. Then a group action of $\Z$ on $G$ by $(n,g) \mapsto
 \gamma^n g$ is a coarse action and the orbit of $e \in G$ is coarse. 
 Thus the action of $\Z$ on the Higson corona $\nu G$ has a fixed
 point. Moreover, if $G$
 is a hyperbolic group, this action extends to the Gromov boundary
 $\partial_g G$ (see Proposition \ref{prop:equivariant_map} and Example
 \ref{exmp:Hyperbolic_group}). Then this action of $\Z$ on $\partial_g G$
 has a fixed point. This is a famous fact on the group action on the
 boundary of hyperbolic groups. 
(c.f. Proposition 10 and Th\'{e}or\`{e}m 30 in Chapter 8 of \cite{MR1086648}.)
\end{example}
\begin{example}
 The wreath product $\Z \wr \Z$ contains $\Z^n$ as a subgroup for any
 integer $n$ (see page 135 of \cite{MR2007488}). Thus for any $n \in
 \N$, there is a coarse action of $\Z^n$ on $\Z \wr \Z$ and induced
 action of $\Z^n$ on $\nu (\Z \wr \Z)$ has a fixed point.
\end{example}
\begin{remark}
 Let $F_2$ be a free group of rank 2. By Lemma \ref{lem:coarse_action},
 the action of $F_2$ on $F_2$ is a coarse
 action and has a coarse orbit.
 However, the induced action of $F_2$ on the Higson corona
 $\nu F_2$ does not have any fixed point 
(see Proposition \ref{prop:F_2_action_on_the_Higson_corona}).
%${}_\nu \Phi_g(x) = x$ for any element $g$ of
% $F_2$ where ${}_\nu \Phi_g$ is a induced map on $\nu F_2$. 
\end{remark}
Let $f \colon X \rightarrow X$ be a coarse map. We call a point $x$ of
$X$ as a {\itshape coarse fixed point} of $f$ if the orbit of $x$,
$\{f^n(x)|n\in \N\}$, is a bounded set.
\begin{corollary}[Coarse version of Brouwer's fixed point theorem]
\label{cor:Brouwer_theorem}
 Let $X$ be a proper metric space. Suppose that $f \colon X \rightarrow X$ is
 an isometry map and $\overline{X}$ is a  coarse compactification of $X$
 such that $f$ extends to a continuous map $\bar{f} \colon \bar{X}
 \rightarrow \bar{X}$. If $f$ does not have any coarse fixed points
 in $X$, then $\bar{f}$ has a fixed point in $\partial X$.
\end{corollary}
This corollary says that an isometry $f$ always has a coarse fixed point
in $X$ or a fixed point in $\partial X$.
\begin{example}
 Let $M$ be a compact path metric space. The {\itshape cone} $CM$ on $M$
 is the quotient space $M\times [0,\infty) / \sim$ where 
 $(x,t) \sim (x',t')$
 if and only if either $x=x'$ and $t=t'$ or $t=t'=0$.
 We define a continuous function 
$\tau\colon [0,\infty) \rightarrow [0,1)$ by
\[
 \tau(t) = \frac{t}{1+t} \;.
\]
 Set $\overline{CM} = M\times [0,1]/\sim $.
 The embedding $CM \rightarrow \overline{CM}$ :
 $(x,t) \mapsto (x,\tau(t))$ gives us a compactifications of $CM$.
 We can define an appropriate metric on $CM$ such that $CM$ is a proper
 metric space. Then $\overline{CM}$ becomes a coarse
 compactification (see Appendix).
 Let $f \colon CM \rightarrow CM$ be an isometry map and we suppose
 that $f$ extends to a continuous map 
$\bar{f} \colon \overline{CM}  \rightarrow \overline{CM}$.
 Then $f$ has a coarse fixed point on $X$ or $\bar{f}$ has
 a fixed point on the boundary $\partial (CM) \cong M$.
% Let $f_M$ be an isometry on the boundary
% $\partial(CM) \cong M$ to itself without any fixed point. 
% If we extend $f_M$ to an isometry map
% $f_{CM}$ on $CM$, then $f_{CM}$ has a coarse fixed point.
 If $M$ is ENR, then Lefschetz fixed point theorem implies that $\bar{f}$ has a
 fixed point since $\overline{CM}$ is contractible. However, Let 
 $g$ be an isometry map on a {\itshape punctured cone}
$CM^\times  = M \times [1/2,\infty)$. We suppose that 
$g$ extends to a continuous map $\bar{g}$
on a coarse compactification $\overline{CM^\times} = M \times [1/2,1]$.
Then $g$ also has a coarse fixed points on $CM^\times $ 
or $\bar{g}$ has a fixed point on the boundary $\partial (CM) \cong M$.
Since $CM^\times$ is homotopic to $M$, Lefschetz fixed point theorem
 does not imply this fact.
\end{example}
\begin{example}
 Set $M=S^{n-1}$. Then $CM$ is homeomorphic to the $n$-dimensional
 Euclidean space $\R^{n}$ and $\overline{CM}$ is homeomorphic to unit
 ball $B^n = \{x \in \R^n| |x| \leq 1\}$. Any isometry map $f\colon \R^n
 \rightarrow \R^n$ which can be extends to continuous map 
 $\bar{f}\colon B^n \rightarrow B^n$ has a coarse fixed point on $X$ or
 a fixed point on $\partial (\R^n) = S^{n-1}$. This corresponds to the
 Brouwer's fixed point theorem.
\end{example}
%%General reference for coarse geometry is \cite{MR2007488}.

%%%Acknowledgments may be included at the end of the Introduction.
\subsubsection*{Acknowledgments}
We would like to thank the participants in the coarse geometry
seminar at Kyoto university, S.Honda, T.Kato, T.Kondo and M.Tsukamoto for
several discussions and useful comments.

\section{Coarse action}

\begin{lemma}
\label{lem:coarse_action}
 Let $G$ be a finitely generated group or $G=\N^k$ with left-invariant
 word metric for some generating set. The action of $G$ on $G$ by the
 left-translation $(g,h) \mapsto gh$ is a coarse action. Furthermore,
 for any point $h$ of $G$, the orbit of $h$ is coarse.
 any orbit of $h \in H$ has a coarse orbit.
\end{lemma}
\begin{proof}
 For given $g \in G$, the map $\Psi_g \colon G \rightarrow G$ by $h \mapsto
 gh$ is an isometry map, so it is a coarse map.
Let $h \in G$ be  given, we consider the orbit map $\Phi_h \colon G
 \rightarrow G$ by $g \mapsto gh$. We denote the word length by
 $|\cdot|$. For any $g, g' \in G$, we have
\[
d(\Phi_h(g),\Phi_h(g')) = d(gh, g'h) 
=|h^{-1}g^{-1}g'h| \leq |g^{-1}g'| + 2 |h|.
%= d(e, h^{-1}g^{-1}g'h) \leq
% d(e,h^{-1}) + d(h^{-1}, h^{-1}g^{-1}g') + d(h^{-1}g^{-1}g',
% h^{-1}g^{-1}g'h') =
%  d(e,g^{-1}g') + 2d(e,h).
\]
%where $|\cdot|$ denotes the word length.
This shows that $\Phi_h$ is a large scale Lipschitz map and hence a
 bornologous map. Let $D$ be a bounded subset of $G$, which is a finite
 set. Since $\Phi_g$ is injective,  
$\sharp \Phi_g^{-1}(D) \leq \sharp D < \infty$. 
 This shows that $\Phi_h$ is a proper map. 
It follows that $\Phi_h$ is a coarse map and the orbit of $h$ is coarse.
\end{proof}
The coarse action of $G$ on $X$ induces the continuous action on the
Higson corona of $X$. Here we recall the definition and basic properties
of the Higson corona. For more details, see Section 2.2 and 2.3 of
\cite{MR2007488}.

Let $X$ be a proper coarse space. Let $f \colon X \rightarrow \C$ be
a bounded continuous function on $X$. 
We denote by $\mathbf{d}f$ the function  
$\mathbf{d}f(x,y) = f(y) - f(x) \colon X \times X \rightarrow \C$. 
We define that $f$ is a {\itshape Higson function}
 if for any controlled set $E \subset X \times X$,
the restriction of $\mathbf{d}f$ to $E$ vanishes at infinity$^*$.
\footnote{$^*$That is,
for any $\epsilon > 0$, there exists a bounded subset $B \subset X$ such
that $|\mathbf{d}f(x,y)| < \epsilon $ for any 
$(x,y) \in E \setminus B \times B$.}
We denote $C_h(X)$ as the set of Higson functions on $X$, which becomes
unital $C^*$-algebra. The compactification $hX$ of X characterized by
$C(hX) = C_h(X)$ is called the {\itshape Higson compactification}.
 Its boundary $hX \setminus X$ is denoted $\nu X$ and is called the
 {\itshape Higson corona} of $X$. 
The following are basic properties of the Higson corona.
\begin{proposition}
 Let $f\colon X \rightarrow Y$ be a coarse map between proper coarse
 spaces. Then $f$ extends to a continuous map ${}_\nu f \colon \nu X
 \rightarrow \nu Y$. If $f,g \colon X \rightarrow Y$ are close, then
 ${}_\nu f = {}_\nu g$.
\end{proposition}
The Higson corona is a huge space (usually not second countable). So it is
useful to find more manageable compactifications. Thus we need to know
the relation between the Higson corona and such compactifications.
\begin{definition}
 Let $X$ be a proper coarse space and $\overline{X}$ be a compactification of
 $X$. We call $\overline{X}$ a {\itshape coarse compactification}
 if the following
 condition is satisfied: For any controlled set $E \subset X \times X$
 of $X$, its closure $\overline{E}$ in $ \overline{X} \times
 \overline{X}$ meets the 
 boundary $\partial (X\times X) = \overline{X} \times \overline{X} \setminus (X
 \times X)$ only at the diagonal, that is, 
\[
 \overline{E} \cap \partial (X \times X) \subset
 \Delta_{\partial X} = \{(\omega, \omega) | \omega \in \partial X\}.
\]
\end{definition}
\begin{proposition}
The Higson compactification  $hX$ is a coarse compactification. Furthermore,
 it is a universal coarse compactification, in the sense that, for any
 coarse compactification $\overline{X}$ of $X$, the identity map 
 $\mathrm{id}\colon X\rightarrow X$ extends
 uniquely to a continuous surjective map $\iota \colon hX \rightarrow
 \overline{X}$.
\end{proposition}
\begin{proposition}
\label{prop:equivariant_map}
 Let $f\colon X \rightarrow X$ be a continuous coarse map and $\overline{X}$
 be a coarse compactification. Suppose that $f$ extends to a continuous
 map $\bar{f} \colon \overline{X} \rightarrow \overline{X}$. Then
 $\bar{f}(\partial X) \subset \partial X $ and the following
 diagram commutes. 
\begin{eqnarray*}
 \xymatrix{
   hX \ar[d]_{\iota} \ar[rr]^{{}_hf=f \cup {}_\nu f} & &
   \ar[d]^{\iota} hX  \\
   \overline{X} \ar[rr]^{\bar{f}} &
    & \overline{X}
}
\end{eqnarray*}
%\begin{remark}
%\end{remark}
\begin{proof}[Proof of Proposition \ref{prop:equivariant_map}]
 By the properness of $f$, we have 
$\bar{f}(\partial X) \subset \partial X $.
Since $f$ is a continuous coarse map, $f$ induces $f^* \colon C_h(X)
 \rightarrow C_h(X)$. Because $f$ extends to the continuous map on
 $\overline{X}$, $f^*$ maps $C(\overline{X})$ to itself. It is clear
 that $f^*$ and the inclusion $\iota ^* \colon C(\overline{X})
 \hookrightarrow C_h(X)$ commute.
\end{proof}
\begin{example}
\label{exmp:Hyperbolic_group}
 Let $X$ be a Gromov hyperbolic space. Then the Gromov
 compactification is a coarse compactification (Lemma 6.23 of
 \cite{MR2007488}). If $G$ is a hyperbolic 
 group and $\gamma \in G$ has an infinite order, then left-translation
 $\gamma \cdot \colon G \rightarrow G$ extends to a continuous map on
 the Gromov boundary $\partial_g G$ 
(see  Chapter 7 and 8 of \cite{MR1086648}).
\end{example}
\end{proposition}
\begin{lemma}
\label{lem:abelian_gr_action_on_itself}
 Let $G$ be $\N^k$ or $\Z^k$. A coarse action of
 $G$ on $G$ by $(g, n)\mapsto g+n$ extends to a trivial action on the
 Higson corona of $G$. Namely, for any $g\in G$, the induced map
${}_\nu\Psi_{g} \colon \nu G \rightarrow \nu G$ is equal to the identity
 map. 
\end{lemma}
\begin{proof}
 Let $g \in G$ be given. For any $n\in G$, $d(n, \Psi_g(n)) = d(n,g+n) =
 |g|$. Then $\Psi_g$ and $\mathrm{id}_G \colon G \rightarrow G$ are
 close. It follows that ${}_\nu\Psi_{g} = \mathrm{id}_{\nu G}$.
\end{proof}
\begin{proof}[Proof of Theorem \ref{th:main_theorem}]
Since the action is coarsely free, there exists $x_0 \in X$ such that
 the $G$-equivariant map $\Phi_{x_0} \colon G \rightarrow X$ by 
$g \mapsto g \cdot x_0$ is a coarse map. Thus $\Phi_{x_0}$ extends to a
 continuous map ${}_\nu \Phi_{x_0}$ on $\nu G$ to $\nu X$.
For each $g \in G$, the coarse
 map $\Psi_{g} \colon X \rightarrow X$ extends to continuous maps 
${}_\nu \Psi_g \colon \nu X \rightarrow \nu X$ and $\overline{\Psi_g}
 \colon \partial X \rightarrow \partial X$. Then we have the following
 commutative diagram:
\begin{eqnarray*}
 \xymatrix{
   \nu G \ar[d]_{{}_\nu \Phi_{x_0}} \ar[rr]^{{}_\nu \Psi_{g} = \mathrm{id}} & &
   \ar[d]^{{}_\nu \Phi_{x_0}} \nu G \\
   \nu X\ar[rr]^{{}_\nu \Psi_{g}} &
    &\nu X\\ 
}
\end{eqnarray*}
By Lemma \ref{lem:abelian_gr_action_on_itself}, the upper right arrow is
the identity map.
We choose any $x' \in \nu G$ and denote 
$x = {}_\nu \Phi_{x_0}(x')$.
Then by the above commutative diagram, we have 
${}_\nu \Psi_{g}(x) = x$.

Moreover, Let $\overline{X}$ be a coarse compactification of $X$.
We suppose that the action of $G$ on $X$ is continuous and extends to
 a continuous action on $\overline{X}$. Then we have the following
 commutative diagram:
\begin{eqnarray*}
 \xymatrix{   \nu X  \ar[d]_{\iota} \ar[rr]^{{}_\nu \Psi_{g}} &
    &\ar[d]^{\iota} \nu X\\ 
   \partial {X}  \ar[rr]^{\overline{\Psi_{g}}} & &\partial X
}
\end{eqnarray*}
Set $z = \iota(x)$. We have $\overline \Psi_g(z) = z$.
\end{proof}
\begin{remark}
 Suppose that $X$ is a proper metric space and $\Phi_{x_0}$ is a continuous
 map. 
 By the results of Dranishnikov; Keesling; Uspenskij
 \cite{MR1607744}, if $\Phi_{x_0}$ is a coarse embedding and 
$\Phi_{x_0}(G)$ is a closed subset, we can show that the map  ${}_\nu
 \Phi_{x_0}$ in the proof of Theorem
 \ref{th:main_theorem} is an embedding. Indeed, Theorem 1.4 of
 \cite{MR1607744} states that the closure of $\Phi_{x_0}(G)$ in $hX$ is
 homeomorphic to the Higson corona $\nu(\Phi_{x_0}(G))$ of $\Phi_{x_0}(G)$,
 whose coarse structure is a bounded coarse
 structure defined by the induced metric from $X$.
 Since $\Phi_{x_0}$ is a coarse embedding, $\nu G$ is
 homeomorphic to $\nu (\Phi_{x_0}(G))$. Thus  ${}_\nu \Phi_{x_0} \colon \nu G
 \rightarrow \nu X$ is an embedding. It follows that there are so many
 fixed points on the Higson corona $\nu X$.
\end{remark} 
\begin{proposition}
\label{prop:F_2_action_on_the_Higson_corona}
 The action of $F_2$ on $\nu F_2$ induced by the left-translation $F_2
 \times F_2 \rightarrow F_2$ does not have any fixed point.
 Namely, there are no point $x$
 of $\nu F_2$ such that $g\cdot x = x$ for any element $g$ of
 $F_2$.
\end{proposition}
\begin{proof}
 Gromov compactification is a coarse compactification, and isometry
 action on $X$ induces the continuous action on the boundary.
 One can
 easily see that if the induced action of $F_2$ on $\nu F_2$ has a fixed
 point, the induced action of $F_2$ on the Gromov boundary $\partial_g
 F_2$ of $F_2$ also has a fixed point. However, we can show that for any
 point $z \in \partial_g F_2$, there exists an element $g$ of $F_2$
 such that $g\cdot z \neq z$.
\end{proof}
\section{Coarse fixed point}
\begin{definition}
 Let $G$ be a finitely generated semi-group acting on $X$. We call 
 a point $x$ of $X$ a {\itshape coarse fixed point} if its orbit 
$G \cdot x = \{g \cdot x| g \in G\} \subset X$ is a bounded set.
\end{definition}
\begin{remark}
 \label{rem:coarse_fixed_point}
If $G$ is an infinite group and $x$ is a coarse fixed point, then the
 orbit of $x$ is not proper.
\end{remark}
We are interested in the problem: When does the converse of Remark
\ref{rem:coarse_fixed_point} holds?
\begin{proposition}
\label{prop:uniformly_discreate}
 Let $X$ be a coarse space such that any bounded subset $D
 \subset X$ is a finite set. Suppose that an abelian semi-group $\N$
 acts on $X$. Then a point $x_0$ of $X$ whose orbit is not proper is a
 coarse fixed point.
\end{proposition}
\begin{proof}
 Since the orbit of $x_0$ is not proper, there exists a bounded set $D
 \subset X$ such that $\sharp \{n \in \N| n \cdot x_0 \in D\} =
 \infty$. Because $\sharp D < \infty$, there exists $m, n \in \N, m > n$
 such that $m\cdot x_0 = n\cdot x_0$. Hence for any $l > m$, there are some
 $k > 0$ and $r = 0, \cdots, m-n-1$ satisfying $l - n = k(m-n) + r$ and thus
 $l\cdot x_0 = r \cdot(n\cdot x_0) $. It follows that $\N \cdot x_0
 \subset \{x_0, 1\cdot x_0, 2\cdot x_0, \cdots, (m-1)\cdot x_0\}$.
\end{proof}
\begin{proposition}
\label{prop:coarse_fixed_point}
 Let $X$ be a proper metric space. Suppose that an abelian
 semi-group $\N$ acts on $X$ as an isometry action.
 Then a point  $x_0$ of $X$ whose
 orbit is not proper is a coarse fixed point. 
\end{proposition}
\begin{proof}
  Since the orbit of $x_0$ is not proper, there exists a bounded set $D
 \subset X$ such that $\sharp \{n\in \N| n \cdot x_0 \in D\} =
 \infty$. We can assume that $x_0$ lies in $D$. We notice that for any 
 point $x$ of the orbit $\N\cdot x_0$, 
 there exist $n(x) \in \N$ such that $n(x) \cdot x$ lies in
 $D$. We define a bounded subset $K \subset X$ by
\[
 K = B(D,1) \cap \N \cdot x_0.
\] 
Here $B(D,1) = \{x \in X| {}^\exists y \in D, d(x,y) < 1 \}$ is the
 1-neighborhood of $D$.
Since $\overline{K}$ is a compact set, there exists $x_1,\dots, x_N \in K$
 such that
\[
  \overline{K} \subset \bigcup_{i = 1}^N B(x_i,1).
\] 
For $0\leq i \leq N$, we denote $T_i = n(x_i)$. The point $T_i \cdot x_i$ lies
 in $D$.
We define a positive number $L$ by
\[
 L = \max_{0\leq i \leq N}\max_{0 \leq a \leq T_i}d(x_0, a\cdot x_i).
\]
We define inductively a sequence $\{i_k\}_{k=0}^\infty$ 
consisting of integers in $ \{1,\dots,N\}$ 
and an increasing sequence $\{S_j\}_{j=0}^{\infty}$ as
 follows:
Let $S_0 = T_{0}$. Since $S_0\cdot x_0 = T_{0}\cdot x_0 \in
 K$, there exists an integer $i_0 \in \{1,\dots,N\}$ such that $S_0\cdot x_0 \in B(x_{i_0},1)$.
 Assume that we have defined $i_0,\dots, i_n$ and $S_0,\dots, S_n$ such that, 
for $0\leq j \leq n$ and $0\leq a < T_{i_j}$, they satisfy
\begin{eqnarray*}
 S_j &=& \sum_{k=0}^{j-1}T_{i_k}, \\
 S_j \cdot x_0 &\in& B(x_{i_j}, 1) ,\\
 (a + S_j)\cdot x_0 &\in& B(x_0, L + 1).
\end{eqnarray*}
Let $S_{n+1} = T_{i_n} + S_{n}$. Since $T_{i_{n}}\cdot x_{i_{n}} \in D$ and 
\[
 d(S_{n+1}\cdot x_0,\, T_{i_n}\cdot x_{i_n}) = d(S_n \cdot x_0,\, x_{i_n}) < 1,
\]
we have $S_{n+1} \cdot x_0 \in B(D,1)$
 and thus $S_{n+1} \cdot x_0$ lies in $K$. Hence there exists
 $i_{n+1}$ such that 
 $S_{n+1} \cdot x_0$ lies in $B(x_{i_{n+1}},1)$. For $0 \leq a <
 T_{n+1}$, we have $d(x_0, a\cdot x_{i_{n+1}}) \leq L$ and 
\[
 d((a + S_{n+1})\cdot x_0,\, a \cdot x_{i_{n+1}}) 
 = d(S_{n+1} \cdot x_0,\, x_{i_{n+1}}) < 1.
\]
It follows that 
$(a + S_{n+1})\cdot x_0$ lies in $B(x_0, L + 1)$.
This shows that for any integer $l > 0$, $l\cdot x_0$ lies in $B(x_0, L+1)$
 and thus $x_0$ is a coarse fixed point.
%%we have $\sum_{_k=1}^{j+1}T_{i_k} \cdot x_0 \in B(D,1)$.
\end{proof}
\begin{remark}
Under the assumption of Proposition \ref{prop:coarse_fixed_point}, if
 $x_0$ is a coarse fixed point, then any point $x$ of $X$ is a coarse fixed
 point.
\end{remark}
If an isometry action does not have any coarse fixed point, the orbits
of any points are proper. On the other hand, the orbits by an isometry
action are always bornologous.
\begin{lemma}
\label{lem:isometry-action_bornologous}
 Let $X$ be a proper metric with an isometry action of $\N$.
 Then the action is a coarse action and any orbits are bornologous.
\end{lemma}
\begin{proof}
 An isometry action is a coarse action. For any given point $x$ of $X$,
 let $L = d(1\cdot x, x)$ be the distance between $1\cdot x$ and $x$. 
Then we have $d((i+1)\cdot x, i\cdot x) = L$ for all integers $i > 0$.
 Hence for any integers $m \geq n > 0$,
\[
  d(\Phi_x(m), \Phi_x(n)) = d(m\cdot x, n\cdot x) \leq
 \sum_{i=n}^{m-1}d((i+1)\cdot x, i\cdot x) = L|m-n|.
\]
Thus $\Phi_x$ is a bornologous map.
\end{proof}
\begin{proof}[Proof of Corollary \ref{cor:Brouwer_theorem}]
 A coarse action of $\N$ on $X$ is defined by $(n,x) \mapsto f^n (x)$.
 This is an isometry action. 
 If $f$ does not have any coarse fixed point,
 then by Lemma \ref{prop:coarse_fixed_point} and Lemma
 \ref{lem:isometry-action_bornologous}, any orbit of this action is
 coarse. Thus Theorem \ref{th:main_theorem} implies that $\bar{f}$ has 
 a fixed point. 
\end{proof}
%%%

\section{More example of coarse action} %delete * to number this section
\subsection*{Adding machine}
Let $T_2 = (V_2, E_2)$ be a binary tree with the set of vertices $V_2$ and
the set edges $E_2$. We identify $V_2$ with the following sets.
\[
 V_2 \cong  \{*\} \sqcup \bigsqcup_{n\geq 1} \{0,1\}^n.
% V_2 \cong \left(\bigsqcup_n \{0,1\}^n\right) \sqcup \{*\}.
% \cong  \Big\{(\cdots,*,\dots,*,i_{n-1},\cdots,i_0) \in \{0,1,*\}^\N |
% i_k\in \{0,1\}\Big\}.
\]
Here $*$ denotes the base point.
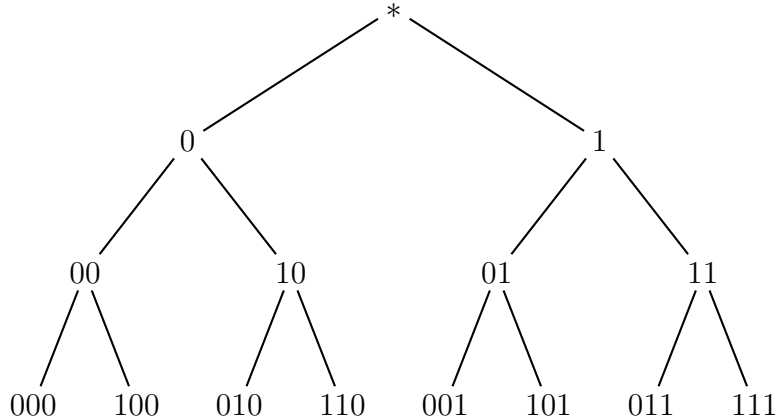
\begin{center}
% \pstree[treemode=R,nodesep=3pt, levelsep=60pt]{\TR{$*$} }{
 \pstree[nodesep=3pt, levelsep=50pt]{\TR{$*$} }{
        \pstree{ \TR{0} }{
                     \pstree{ \TR{00}}{
                                   \TR{000}
                                   \TR{100}
                                             }
                     \pstree{ \TR{10}}{
                                   \TR{010}
                                   \TR{110}
                                             }
                                                }
        \pstree{ \TR{1} }{

                     \pstree{ \TR{01}}{
		                   \TR{001}
				   \TR{101}
				             }
                     \pstree{ \TR{11}}{
		                   \TR{011}
				   \TR{111}
				             }
			     		        }
                                                   }
\end{center}
$*$ and $0$ (resp. $1$) can be joined by the edge $e_{*,0}$
(resp. $e_{*,1}$).
Let $x=(i_{n-1},\dots,i_0)$ and $y=(j_{n},\dots,j_0)$ be two
vertices. $x$ and $y$ can be joined by the edge $e_{x,y}$ if and only if
$i_0 = j_0, \dots, i_{n-1} = j_{n-1}$. $T_2$ has the usual metric and is
a geodesic space. $V_2$ has the induced metric from $T_2$.

We construct a coarse action of $\N$ on $V_2$. Let $x$ be a vertex.
If $x = (i_n,\dots, i_0) \neq (1,\dots,1)$, we define $1\cdot x =
(j_{n-1}, \dots, j_0)$ with 
\[
 \sum_{k=0}^{n-1} j_k 2^k =  \sum_{k=0}^{n-1} i_k 2^k + 1.
\]
If $x = (\overbrace{1,\dots,1}^{n})$, we define
 $1\cdot x = (\overbrace{1,0,\dots,0}^{n+1})$. Finally we define 
$1\cdot * = (0)$.

The Gromov product of $x=(i_{n-1},\dots,i_0)$ and
$y=(j_{m-1},\dots,j_0)$ with a base point $*$ is defined to be
\[
 (x|y) = \frac{1}{2}\{d(x,*) + d(y,*) - d(x,y)\}.
\]
We notice that  $(x|y) = r$ if and only if $i_0 = j_0, \dots, i_{r-1} = j_{r-1}$
and $i_r \neq j_r$.
\begin{lemma}
 $1\cdot \colon V_2 \rightarrow V_2$ is a coarse map.
\end{lemma}
\begin{proof}
 Let $x=(i_{n-1},\dots, i_{0})$ and $y=(j_{m-1},\dots,j_{0})$ be two
 vertices. Set $(x|y) = r$. Then $i_0=j_0,\dots,i_{r-1}=j_{r-1}$. 
 It is easy to see that $(1\cdot x | 1\cdot y) \geq r$.
 Hence we have
\begin{eqnarray*}
 d(1\cdot x, 1\cdot y) &=& 
d(1\cdot x, *) + d(1\cdot y,*) -2 (1\cdot x | 1\cdot y)\\
&\leq& (n+1) + (m+1) -2r\\
&=& d(x,y) + 2.
\end{eqnarray*}
 Thus $x \mapsto 1\cdot x$ is a Large scale Lipschitz map.
 Since the map $1\cdot$ is injective and every bounded set of $V_2$
 is a finite set, we have the map is proper.
\end{proof} 
This action can be extended to a continuous action on the Gromov boundary
$\partial_g V_2$ of $V_2$. Here we can identify $\partial_g V_2$ as the
Cantor set $\{0,1\}^\N$ with a metric $d$ defined to be
\[
 d(x,y) = 2^{-(x|y)}.
\]
Here $x=(i_k)_{k=0}^\infty$, $y=(j_k)_{k=0}^\infty$. $(x|y) = r$ if
and only if $i_0 = j_0, \cdots , i_{r-1} = j_{r-1}$ and $i_r \neq j_r$.
\begin{proposition}
 The action of $\N$ on $\partial_g V_2$ is minimal. Indeed, for any
 point $x$ of $\partial_g V_2$, its orbit $\N \cdot x$ is dense
 in $\partial_g V_2$.
\end{proposition}
\begin{proof}
 Let $x\in \partial_g V_2$ be given. For any point $y$ of 
$\partial_g V_2$ and $\epsilon > 0$, we will show that there exists an
 integer $n$ such that $d(n\cdot x, y) < \epsilon$.
 Set $N > - \frac{\log \epsilon }{\log 2}$ and $x=(i_k)_{k=0}^\infty$. We
choose $a = 2^{N+1} - \sum_{k=0}^{N}i_k2^k$, then we have 
\[
 a \cdot x = (\cdots, \overbrace{0,\cdots,0}^{N+1}).
\]
 Set $y = (j_k)_{k=0}^\infty$ and $b = \sum_{k=0}^{N}j_k2^k$. Then we
have
\[
 (a + b) \cdot x = (\cdots, j_N,\cdots,j_0).
\]
It follows that $d((a+b) \cdot x, y) < 2^{-N} < \epsilon$.
\end{proof}
Especially, the action $\N$ on $\partial_g V_2$ does not have any fixed
points. On the other hand, the orbit of $x \in V_2$ is not a coarse
fixed point. Thus by Proposition \ref{prop:uniformly_discreate} and
Theorem \ref{th:main_theorem}, the orbit of $x \in V_2$ can never be
bornologous.

\section*{Appendix}
\subsection*{Metric of the cone and its compactification}
 Let $M$ be a compact path metric space. The cone $CM$ on $M$ is the
 quotient space 
 $M\times [0,\infty) / \sim$ where $(x,t) \sim (x',t')$ if and only if
 either $x=x'$ and $t=t'$ or $t=t'=0$. According to Roe (section 3.6 of
 \cite{MR1147350}), we construct a metric on $CM$. Let 
$\lambda \colon [0,\infty) \rightarrow [0,\infty)$ be a continuous
 function with $\lambda(t) = 0$ if and only if $t=0$. 
If $\gamma$ is a path in $CM$, we define its
 $\lambda$-length, $l_\lambda(\gamma)$, to be 
\[
 \sup \left\{\sum_{j=0}^{n-1}\Big(|t_j-t_{j+1}| 
 + \max\{\lambda(t_j), \lambda(t_{j+1})\}d(x_j,x_{j+1})\Big)\right\}
\]
where the supremum is taken over all finite sequence $(x_j,t_j)_{j=0}^n$
 of points on the path $\gamma$ (with $(x_0,t_0)$ and $(x_n,t_n)$ being
 two end points).
We define a metric $d_\lambda$ on $CM$ by $d_\lambda((x,t),(x',t')) =
 \inf l_\lambda(\gamma)$ where the infimum is taken over all path
 joining $(x,t)$ and $(x',t')$. The metric $d_\lambda$ is compatible
 with the topology of $CM$ and it becomes a proper metric space
 (Proposition 3.47 of \cite{MR1147350}).
\begin{proposition}
 Suppose that $\lambda$ is an increasing, unbounded function, then the
 compactification $\overline{CM} = M \times [0,1]/ \sim$ is a coarse
 compactification. 
\end{proposition}
\begin{proof}
 Let $E \subset CM \times CM$ be a controlled set and $r_E$ be a positive
 number such that 
\[
 \sup\{d_\lambda(\mathbf{x},\mathbf{y})
 |(\mathbf{x},\mathbf{y})\in E\} < r_E .
\]
Let $\{(\mathbf{x}_n,\mathbf{y}_n)\} \subset E$ be a sequence such that
 $\bf{x}_n$ and $\bf{y}_n$ tends to $x$ and $y$ in $\partial(CM) \cong M$
 respectively. We denote $\mathbf{x}_n = (x_n, t_n)$ and 
$\mathbf{y}_n = (y_n, s_n)$. It is enough to show that $x = y$ (see
 Theorem 2.27 of \cite{MR2007488}). 
Suppose that $x \neq y$. Since $x_n$ and $y_n$ converge to $x$ and $y$
 respectively, for any $\epsilon > 0$ there exists $N_1 > 0$ such that,
 for all $n > N_1$, $x_n$ and $y_n$ satisfy 
 $d(x_n,y_n) > d(x,y) - \epsilon$.
Because $\lambda$ is an increasing, unbounded function, there exists
 $N_2 > 0$ such that, for all $t > N_2$,
%% the value of $\lambda$ at $t$ satisfies 
\[
 \lambda(t) > \frac{r_E}{d(x,y)-\epsilon}.
\]
Since $\bf{x}_n$ tends to infinity, there exists $N_3 > 0$ such that,
 for all $n > N_3$, the second component $t_n$ of $\bf{x}_n$ satisfies
$t_n > N_2$.
Thus, if we choose an integer $n$ satisfying $n > \max\{N_1,N_3\}$, then
 we have $d_\lambda(\mathbf{x}_n,\mathbf{y}_n) > r_E$, this contradicts the
 definition of $r_E$.
\end{proof}
%%%%%%%%%%%% References %%%%%%%%%%%%%
%%
%<Author name> is written as Initial of Given Name, and Family Name.
%<Title> is written in roman letters.
%<Journal name> should be abbreviated according to 
% the MR Serials Abbreviations List of Mathematical Reviews:
% (Abbreviations of Names of Serials; http://www.ams.org/mr-database)
%For <Pages>, use en-dash "--" between page numbers.
%%

\bibliographystyle{amsplain}
%\bibliography{/DirUsers/tomo_xi/Library/tex/books,/DirUsers/tomo_xi/Library/tex/math}

\providecommand{\bysame}{\leavevmode\hbox to3em{\hrulefill}\thinspace}
\providecommand{\MR}{\relax\ifhmode\unskip\space\fi MR }
\providecommand{\MRhref}[2]{%
  \href{http://www.ams.org/mathscinet-getitem?mr=#1}{#2}
}
\providecommand{\href}[2]{#2}

\bigskip
%%%%%%%%%%%% Authors' addresses %%%%%%%%%%%%%
\address{ % First Author
Department of Mathematics\\
Kyoto University \\
Kyoto 606-8502\\
Japan
}
{tomo\_xi@math.kyoto-u.ac.jp}
\end{document}